\documentclass[12pt]{article}
\usepackage{amsmath, amsthm}
\usepackage{amscd}
\usepackage{amssymb}
\usepackage{array}
\usepackage{enumerate}
\usepackage{enumitem}
\usepackage{MnSymbol}

\newtheorem{thm}{Theorem}[section]
\newtheorem{theorem}[thm]{Theorem}

\newtheorem{lemma}[thm]{Lemma}
\newtheorem{proposition}[thm]{Proposition}

\theoremstyle{definition}
\newtheorem{definition}[thm]{Definition}

\newtheorem{remark}[thm]{Remark}

\begin{document}




\newcommand{\id}{\mathrm {id}}
\newcommand{\R}{\mathbb{R}}
\newcommand{\C}{\mathbb{C}}
\newcommand{\Z}{\mathbb{Z}}
\newcommand{\N}{\mathbb{N}}
\newcommand{\bD}{\mathbb{D}}
\newcommand{\bG}{\mathbb{G}}
\newcommand{\bP}{\mathbb{P}}
\newcommand{\bA}{\mathbb{A}}
\newcommand{\g}{\mathfrak{G}}
\newcommand{\e}{\epsilon}
\newcommand{\cA}{\mathcal{A}}
\newcommand{\cB}{\mathcal{B}}
\newcommand{\cC}{\mathcal{C}}
\newcommand{\cD}{\mathcal{D}}
\newcommand{\cI}{\mathcal{I}}
\newcommand{\cL}{\mathcal{L}}
\newcommand{\cO}{\mathcal{O}}
\newcommand{\cG}{\mathcal{G}}
\newcommand{\cJ}{\mathcal{J}}
\newcommand{\cF}{\mathcal{F}}
\newcommand{\cP}{\mathcal{P}}
\newcommand{\cU}{\mathcal{U}}
\newcommand{\ep}{\mathcal{E}}
\newcommand{\E}{\mathcal{E}}
\newcommand{\cH}{\mathcal{O}}
\newcommand{\cV}{\mathcal{V}}
\newcommand{\cPO}{\mathcal{PO}}
\newcommand{\cHol}{\mathrm{Hol}}
\newcommand{\cp}{\mathcal{P}}

\newcommand{\rAut}{\mathrm{Aut}}
\newcommand{\rC}{\mathrm{C}}
\newcommand{\rCOSp}{\mathrm{COSp}}
\newcommand{\rGL}{\mathrm{GL}}
\newcommand{\rinM}{\mathrm{\underline{M}}}
\newcommand{\rSU}{\mathrm{SU}}
\newcommand{\rSL}{\mathrm{SL}}
\newcommand{\rPC}{\mathrm{PC}}
\newcommand{\rPSL}{\mathrm{PSL}}
\newcommand{\rPGL}{\mathrm{PGL}}
\newcommand{\PGL}{\mathrm{PGL}}
\newcommand{\rSC}{\mathrm{SC}}
\newcommand{\rSO}{\mathrm{SO}}
\newcommand{\rSpO}{\mathrm{SpO}}
\newcommand{\rOSp}{\mathrm{OSp}}
\newcommand{\rSpin}{\mathrm{Spin}}
\newcommand{\rsl}{\mathrm{sl}}
\newcommand{\rM}{\mathrm{M}}
\newcommand{\rdiag}{\mathrm{diag}}
\newcommand{\rP}{\mathrm{P}}
\newcommand{\rdeg}{\mathrm{deg}}
\newcommand{\pt}{\mathrm{pt}}
\newcommand{\red}{\mathrm{red}}

\newcommand{\bm}{\mathbf{m}}

\newcommand{\M}{\mathrm{M}}
\newcommand{\End}{\mathrm{End}}
\newcommand{\Hom}{\mathrm{Hom}}
\newcommand{\inHom}{\mathrm{\underline{Hom}}}
\newcommand{\diag}{\mathrm{diag}}
\newcommand{\rspan}{\mathrm{span}}
\newcommand{\rank}{\mathrm{rank}}
\newcommand{\Gr}{\mathrm{Gr}}
\newcommand{\ber}{\mathrm{Ber}}
\newcommand{\asalg}{\mathrm{(asalg)}}
\newcommand{\csalg}{\mathrm{(csalg)}}
\newcommand{\str}{\mathrm{str}}
\newcommand{\Sym}{\mathrm{Sym}}
\newcommand{\tr}{\mathrm{tr}}
\newcommand{\defi}{\mathrm{def}}
\newcommand{\Ber}{\mathrm{Ber}}
\newcommand{\spec}{\mathrm{Spec}}
\newcommand{\sschemes}{\mathrm{(sschemes)}}
\newcommand{\sschemeaff}{\mathrm{ {( {sschemes}_{\mathrm{aff}} )} }}
\newcommand{\rings}{\mathrm{(rings)}}
\newcommand{\Top}{\mathrm{Top}}
\newcommand{\sarf}{ \mathrm{ {( {salg}_{rf} )} }}
\newcommand{\arf}{\mathrm{ {( {alg}_{rf} )} }}
\newcommand{\odd}{\mathrm{odd}}
\newcommand{\alg}{\mathrm{(alg)}}
\newcommand{\grps}{\mathrm{(grps)}}
\newcommand{\sa}{\mathrm{(salg)}}
\newcommand{\sets}{\mathrm{(sets)}}
\newcommand{\smflds}{\mathrm{(smflds)}}
\newcommand{\mflds}{\mathrm{(mflds)}}
\newcommand{\SA}{\mathrm{(salg)}}
\newcommand{\salg}{\mathrm{(salg)}}
\newcommand{\salgk}{\mathrm{(salg)_k}} 
\newcommand{\varaff}{ \mathrm{ {( {var}_{\mathrm{aff}} )} } }
\newcommand{\svaraff}{\mathrm{ {( {svar}_{\mathrm{aff}} )}  }}
\newcommand{\ad}{\mathrm{ad}}
\newcommand{\Ad}{\mathrm{Ad}}
\newcommand{\pol}{\mathrm{Pol}}
\newcommand{\Lie}{\mathrm{Lie}}
\newcommand{\Proj}{\mathrm{Proj}}
\newcommand{\rGr}{\mathrm{Gr}}
\newcommand{\rFl}{\mathrm{Fl}}
\newcommand{\rPol}{\mathrm{Pol}}
\newcommand{\rdef}{\mathrm{def}}
\newcommand{\Aut}{\mathrm{Aut}}
\newcommand{\rOsp}{\mathrm{Osp}}
\newcommand{\SUSY}{\mathrm{SUSY}}

\newcommand{\fsl}{\mathfrak{sl}}

\newcommand{\uspec}{\underline{\mathrm{Spec}} \,}
\newcommand{\uproj}{\mathrm{\underline{Proj}}}
\newcommand{\uM}{\mathrm{\underline{M}}}

\newcommand{\sym}{\cong}

\newcommand{\al}{\alpha}
\newcommand{\be}{\beta}
\newcommand{\lam}{\lambda}
\newcommand{\de}{\delta}
\newcommand{\ttau}{\tilde \tau}
\newcommand{\D}{\Delta}
\newcommand{\s}{\sigma}
\newcommand{\lra}{\longrightarrow}
\newcommand{\ga}{\gamma}
\newcommand{\ra}{\rightarrow}
\newcommand{\Der}{\mathrm{Der}}
\newcommand{\im}{\mathrm{im}}
\newcommand{\Ker}{\mathrm{ker}}
\newcommand{\NOTE}{\bigskip\hrule\medskip}
\newcommand{\tg}{\widehat{g}}

\medskip

\centerline{\Large \bf The Projective Linear Supergroup and the}

\medskip
\centerline{\Large \bf 
SUSY-preserving automorphisms of $\bP^{1|1}$}

\bigskip

\centerline{ R. Fioresi$^\dagger$, S. D. Kwok$^\star$}

\medskip
\centerline{\it $^\dagger$ Dipartimento di Matematica, Universit\`{a} di
Bologna }
 \centerline{\it Piazza di Porta S. Donato, 5. 40126 Bologna. Italy.}
\centerline{\footnotesize e-mail: rita.fioresi@UniBo.it}

\medskip
\centerline{\it $^\star$ Mathematics Research Unit, University of Luxembourg}
 \centerline{\it 6, Rue Richard Coudenhove-Kalergi, L-1359, Luxembourg}
\centerline{{\footnotesize e-mail: 
stephen.kwok@uni.lu}}

\begin{abstract}
The purpose of this paper is to describe the projective linear
supergroup, its relation with the automorphisms
of the projective superspace and to determine the supergroup of
SUSY preserving automorphisms of  $\bP^{1|1}$.
\end{abstract}

\section{Introduction}

The works of Manin \cite{ma1,ma2} and more recently of Witten et al.
\cite{witten, dw1} have drawn attention to projective supergeometry and
more specifically to SUSY curves and their moduli superspaces.

\medskip
In this paper we study the automorphisms of the projective 
superspace $\bP^{m|n}$
and its SUSY-preserving subsupergroup.
We start by defining the projective linear 
supergroup $\rPGL_{m|n}$, 
using the functor of points formalism, and then we show that
this supergroup functor is indeed  representable, that
is, it is the functor of points of  a superscheme. We achieve this by 
realizing $\rPGL_{m|n}$ as a closed subsupergroup scheme of 
$\rGL_{m^2+n^2|2mn}$, mimicking the ordinary procedure. 

In relating this supergroup scheme to
the automorphism  supergroup of $\bP^{m|n}$
we encounter a difficulty, not present in the ordinary setting, namely the
fact that the Picard group of the projective superspace is not known in 
general and involves some difficulties. This is a consequence of the fact that
the supergroup of automorphism of the projective superspace is larger
than $\rPGL_{m|n}$ for $n>1$. 
Neverthless, going to the special case of $n=1$, we
are able to give quite explicitly the projective linear supergroup
and to prove it coincides with the automorphisms of the projective superspace.

The question of singling out the SUSY-preserving automorphisms inside this supergroup was already settled over the complex field by Manin \cite{ma1} and Witten \cite{witten}, we extend their considerations to an arbitrary algebraically closed field $k$, $\mathrm{char}(k) \neq 2$, and provide some extra details of their proofs.

\medskip
The organization of this paper is as follows. In Sec. \ref{pmn-sec}
we start by reviewing some generally known facts on the projective superspace 
and its functor of points to establish our notation.
We then discuss line bundles and projective morphisms, proving,
in Prop. \ref{linebundles-obs}, 
that the Picard group of $\bP^{m|1}$ is $\Z$. To our
knowledge this result is new and gives insight into projective 
supergeometry.
In Sec. \ref{pgl-sec} we define the
projective linear supergroup in terms of functor of points
and we prove its representability by realizing it as a closed subsuperscheme
of the general linear supergroup. Then, in Sec. \ref{aut-sec} we
prove that the projective linear supergroup is the supergroup
of automorphisms of the projective superspace in the case of one odd
dimension. Though the approach in both Sec. \ref{pgl-sec} and \ref{aut-sec}
resembles closely the ordinary one, the results are novel in the
supergeometric context.
In Sec. \ref{susy-sec}, we use the machinery developed previously to
prove that the subsupergroup of $\Aut(\bP^{1|1})$ of SUSY preserving automorphisms of $\bP^{1|1}$ consists precisely of the irreducible component 
$(\mathrm{SpO}_{2|1})^0$ of the $2|1$-symplecticorthogonal supergroup 
$\mathrm{SpO}_{2|1}$ containing the identity. This section is a generalization
of the claims made by Manin in \cite{ma1} regarding complex supergeometry
and provides proofs for such claims for a generic algebraically closed field.

\medskip

{\bf Acknowledgements.} We are indebted to Prof. D. Gaitsgory for
clarifying to us the structure of line bundles over $\bP^n_A$
in the ordinary setting.
We also thank Prof. L. Migliorini for helpful discussions. 
We are also grateful to the anonymous Referee for his/her suggestions 
and remarks on the paper.

\section{The projective superspace $\bP^{m|n}$ }
\label{pmn-sec}

In this section we want to recall different, but equivalent definitions
of projective superspace and we describe the line bundles on it.
For all of our notation and main definitions
of supergeometry,
we refer the reader to \cite{ma2, dm, ccf}.

\medskip

Let $k$ be our ground ring. 

\medskip

We recall that, by definition, the functor of points of a superscheme
$X=(|X|, \cO_X)$ is the functor:
$$
X:\sschemes^o \lra \sets, \quad X(S)=\Hom_{\sschemes}(S, X),\quad
X(\phi)(f)=f \circ \phi
$$
where $\sschemes$ denotes the category of superschemes
(it is customary to use the
same letter for $X$ and its functor of points).
Equivalently (see \cite{ccf} Ch. 10),
we can view the functor of points of $X$ as: 
$X: \salg \lra \sets$:
$$
\, X(R)=\Hom_{\sschemes}(\uspec\, R, X),\quad
X(\phi)(f)=f \circ \uspec(\phi)
$$
where $\salg$ denotes the category of superalgebras (over $k$),
(we shall use the same letter also for this functor).
In fact the functor of points of
a superscheme is determined by its behaviour on the affine superscheme
subcategory, which in turn is equivalent to the
category of superalgebras (see \cite{ccf} Ch. 10, Theorem 10.2.5).
If $X=\uspec \cO(X)$, that is $X$ is affine, we have
that 
$$
X(R)=\Hom_{\sschemes}(\uspec\, R, X)=\Hom_\salg(\cO(X),R)
$$ 
where
$\cO(X)$ denotes the superalgebra of global sections of the sheaf of
superalgebras $\cO_X$. 
We say that $X(R)$ are the \textit{$R$-points} of the superscheme $X$.

\medskip
The algebraic superscheme $\bP^{m|n}$ is defined as the patching of
the $m+1$ affine superspaces $U_i = \uspec \cO(U_i)$, with 
$\cO(U_i)=\uspec k[x_{0}^i,$ $\dots$, 
$\widehat{x_i^i}$, $\dots$,  $x_m^i,\xi_1^i, \dots, \xi_n^i]$
through the change of charts: 
\begin{equation} \label{changechart}
\begin{array}{cccc}
\phi_{ij}:& \cO(U_j)[(x_i^j)^{-1}] & \mapsto & \cO(U_i)[(x_j^i)^{-1}] \\ 
&x_k^j & \mapsto & {x_k^i}/{ x_j^i} \\ 
& x_i^j & \mapsto & {1}/{ x_j^i}\\ 
&\xi_k^{ j}  & \mapsto & {\xi_k^i}/{ x_j^i} 
\end{array}
\end{equation}
(as usual 
$\widehat{x_i^i}$ means that we are omitting the indeterminate $x_i^i$).
Notice that $\cO(U_j)[(x_i^j)^{-1}]$ is the superalgebra
representing the open subscheme $U_j \cap U_i$ of $U_j$ (and similarly
for  $\cO(U_i)[(x_j^i)^{-1}]$).

\begin{proposition}\label{fopts-proj}
The $R$-points of $\bP^{m|n}$, $R \in \salg$ are
given equivalently by:
\begin{enumerate}
\item
$$
\begin{array}{c}
\bP^{m|n}(R)=\{\al: R^{m+1|n} \lra L, \, $R$\hbox{-linear, surjective}\}\big/ \, 
\sim, \\ \\
\bP^{m|n}(\psi): R^{m+1|n} \otimes_R T \lra L \otimes_R T
\end{array}
$$
where $L$ is locally free of rank $1|0$, $\psi: R \lra T$ and 
$\al: R^{m+1|n} \lra L \sim \al': R^{m+1|n} \lra L'$ if and only if 
$ker(\al)=ker(\al')$ (or equivalently, $\al \sim \al'$ if 
they differ by an automorphism of $L$ by multiplication of
an element in $R^\times$).
\item 
$$
\begin{array}{c}
\bP^{m|n}(R)=\{\al: L \hookrightarrow R^{m+1|n} \, 
$R$\hbox{-linear, injective}\},
\\ \\ \bP^{m|n}(\psi): L \otimes_R T  \lra   R^{m+1|n} \otimes_R T
\end{array}
$$
where $L$ is locally free of rank $1|0$.
\end{enumerate}

Let $\cO_S^{m+1|n}= \cO_S \otimes 
k^{m+1|n}$. The $S$-points of $\bP^{m|n}$, $S \in \sschemes$ are
given equivalently by:
\begin{enumerate}[label=(\alph*)]
\item
$
\bP^{m|n}(S)=\{\al: \cO_S^{m+1|n} \lra \cL,  \hbox{surjective} \}\big/ \,\sim$, 
$\bP^{m|n}(\psi): (\psi^*\cO_S)^{m+1|n} 
\lra  \psi^*(\cL)
$
where $\psi:T \lra S$,
$\cL$ is a line bundle on $S$ (of rank $1|0$) and 
$\al: \cO_S^{m+1|n} \lra \cL \sim \al': \cO_S^{m+1|n} \lra \cL':$ if and only if 
$ker(\al)=ker(\al')$ (or equivalently, $\al \sim \al'$ if 
they differ by an automorphism of $\cL$ by multiplication of
an element in $\cO_S^\times$).
\item 
$
\bP^{m|n}(S)=\{\al: \cL \hookrightarrow \cO_S^{m+1|n}\}, \quad
\bP^{m|n}(\psi):\psi^*\cL 
\lra (\psi^*\cO_S)^{m+1|n} 
$ 
\end{enumerate}

\end{proposition}

\begin{proof} The proof relative to (1) and (a) 
works as in the ordinary setting 
and it is detailed in \cite{ccf} Ch. 10. The equivalence with (2) and (b)
is immediate. The equivalence (1) and (2) is essentially the
same as in the ordinary setting (see \cite{eh} Ch. III, Sec. 2
Prop. III-40, Cor. III-42).
\end{proof} 

For every $A \in \salg$, let $\salg_A$ denote the category of 
superalgebras over $A$. We will need to consider also 
$\bP^{m|n}_A$ that is the projective superspace
over a base $A \in \salg$. This means that we are considering the superscheme
obtained by patching the affine superspaces 
$U_i = A[x_j^{i}, \xi_k^{ i}]$,
$i,j=0, \dots, m$, $j \neq i$, $k=1, \dots, n$ 
as above. 
For example, in the Case (2) of Prop. \ref{fopts-proj} each
of the  $T$-points, $T \in \salg_A$, is identified with 
a morphisms $\al:L \lra T^{m+1|n}$ of 
$A$-modules, where $L$ and $T^{m+1|n}$ are $T$-modules which become
$A$-modules via the map $\phi:A \lra T$:  
\begin{equation} \label{proj-over-base}
\bP_A^{m|n}(T)=\Hom_{\sschemes_A}(\uspec \, T,\bP_A^{m+1|n})=
\{
\al:L \hookrightarrow T^{m+1|n} \}
\end{equation}
Notice that the functor of points of $\bP_A^{m|n}$ is defined on the
category of $A$-superalgebras or equivalently on the category of
$A$-superschemes (that is superschemes equipped with a morphism
to the superscheme $\uspec A$ and morphisms compatible with it).

We leave to the reader the generalization of the other cases of
Prop. \ref{fopts-proj} since it is straightforward. 

\medskip
We end this section with some observations on line bundles
and morphisms on $\bP_A^{m|n}$. We start with a
result completely similar to the
ordinary counterpart, that we leave
to the reader as a simple exercise (see also \cite{ccf} Ch. 9).

\begin{proposition}\label{linebundles-end}
We have a bijective correspondence between the following:
\begin{enumerate}
\item 
The set of equivalence classes of $m+n+2$-tuples 
$(L, s_0, \dots, s_m, \sigma_1, \dots, \sigma_n)$, where $L$ is a line bundle on $\bP_A^{m|n}$ globally generated by the global sections $s_0, \dots, s_m, \sigma_1, \dots, \sigma_n$ of $L$, under the relation $(L, s_0, \dotsc s_m, \sigma_1, \dotsc \sigma_n) \sim (L, s'_0, \dots s'_m, \sigma'_1, \dotsc \sigma'_n)$ if and only if there exists some
$c \in \mathcal{O}(\bP_A^{m|n})^*_0$
such that $s'_i = cs_i, \sigma'_i = c\sigma_i$ for all $i$.

\item 
The set of $A$-morphisms $\bP_A^{m|n} \to \bP_A^{m|n}$. 

\end{enumerate}
\end{proposition}

In the ordinary setting we have that a line bundle on
$\bP^{m}_A$ is of the form $\cO(n) \otimes \cL$, 
where $\cL$ is a line bundle on $\uspec A$. This non trivial fact is still
true in supergeometry for $\bP_A^{m|1}$, and it will turn out to be crucial in our
treatment.

\begin{proposition}\label{linebundles-obs}
Every line bundle on $\bP_A^{m|1}$ is isomorphic to $\cO(n) \otimes \cL$, 
where $\cL$ is a line bundle on $\uspec A$.
\end{proposition}

\begin{proof}
A line bundle on $\bP_A^{m|1}$ is determined once we know its transition 
functions, say $g_{ij} \in \cO_{\bP_A^{m|1}}(U_i \cup U_j)^*_0$,
which are even. We then need to prove that any such set of transition 
functions is equivalent, up to a coboundary, to a set of transition 
functions for a line bundle of the form
$\cO(n) \otimes \cL$, for $\cL$ a line
bundle on $\uspec A$. In other words we need to show
$$
h_i|_{U_i \cap U_j} \, g_{ij} \, h_j^{-1}|_{U_i \cap U_j} \, = \,  
(x^i_j)^n, \qquad h_i \in \cO_{\bP_A^{m|1}}(U_i)_0^* 
$$

Notice that
$$
\cO_{\bP_A^{m|1}}(U_p)^*\,=\,(A[x_k^p,\xi^p])_0^*\,=\,(A[\xi^p][x_k^p])_0^*,
\qquad p=i,j. 
$$
Since $\phi_{ij}(\xi^j)=\frac{\xi^i}{x^i_j}$, $\phi_{ij}(x^j_i)=1/x^i_j$ and 
$\phi_{ij}(x^j_k)=\frac{x^i_k}{x^i_j}$ 
($\phi_{ij}$ being the change of chart as in (\ref{changechart})),
we can view the restrictions of the $h_p$'s ($p=i,j$) to $U_i \cap U_j$, 
through this
identification, as both belonging to $(A[\xi^i][x_j^i, (x^i_j)^{-1}])_0^*$.
We now apply the classical result and obtain $h_p'\in (A[\xi^i][x_j^i,
(x^i_j)^{-1}])_0^*$
such that
$$
h_i'g_{ij}(h_j')^{-1} = (x^i_j)^n.
$$
 The $h_p'$'s thus obtained are not yet the sections we want; 
since the odd dimension is one by hypothesis, the most general possible form 
for $h_j'$ is
$$
h_j'=a_0+\al_0\xi^i+\sum_K a_Kx_K^i(x^i_j)^{-|K|}+\sum_L \al_Lx_L^i(x^i_j)^{-|L|}\xi^i+
\sum_k \be_k(x^i_j)^{-k}\xi^i
$$
where $K$ and $L$ are multiindices,  $K=(k_1, \dots ,k_r)$, $k_l\neq j$
($r \in \N)$ and $x_K^i:=x^i_{k_1} \dots x^i_{k_r}$ (similarly for $L$).

In order to eliminate the term $\al_0\xi^i$ which is not well defined on $U_j$, we define:
$$
h_i:=(a_0+\al_0\xi^i)h_i', \qquad h_j:=(a_0^{-1}-a_0^{-2}\al_0\xi^i)h_j'.
$$
and this gives the required sections.
\end{proof} 

Notice that it was absolutely fundamental for our argument
that there is only one odd dimension. 
This calculation will give us key information when we want to 
determine the automorphism supergroup of the projective linear
supergroup.

\section{The Projective Linear Supergroup} 
\label{pgl-sec}

In 
this section we want to define 
the supergroup functor of the projective linear supergroup and 
to show it is representable by producing an embedding of it
as a closed subgroup into the general linear supergroup.

\medskip
Let $\uM_{m|n}(R)$ denote the associative superalgebra of supermatrices of 
order $m|n$ by $m|n$ with entries in a commutative superalgebra $R$. 
More intrinsically, $\uM_{m|n}(R) = \underline{\End}_R(R^{m|n})$.

\begin{definition}
The {\it automorphism supergroup of supermatrices} 
is the supergroup functor $\rAut(\uM_{m|n}): \salg \lra \grps$
\begin{align*}
[\rAut(\uM_{m|n})](R) := \{&f: \uM_{m|n}(R) \to \uM_{m|n}(R) \, | \\
&f \text{ is an $R$-superalgebra automorphism}\}.
\end{align*}
In analogy with the ordinary setting we also will call this
supergroup functor the \textit{projective linear supergroup} and denote
it with $\rPGL(m|n)$.
\end{definition}

Since $\uM_{m|n}(R)$ is itself a free $R$-module of rank $M|N$, 
where $M = m^2 + n^2$ and $N = 2mn$, $\rAut(\uM_{m|n})$ is a subfunctor of 
$\rGL_{M|N}$ in a natural way. We want to prove this is the functor of 
points of a closed subsuperscheme of $\rGL_{M|N}$.
Before proceeding we need a lemma characterizing the morphisms of
the superalgebra of supermatrices.

\begin{lemma} \label{matlemma}

\begin{enumerate}
\item An $R$-linear parity preserving map
$\psi:\uM_{m|n}(R)$ $\lra$ \break $\uM_{m|n}(R)$ is a morphism of
the superalgebra of supermatrices $\uM_{m|n}(R)$ if and only if
\begin{enumerate}
\item $\psi(\id)=\id$; 
\item $\psi(e_{ij})\psi(e_{kl})=\delta_{kj}\psi(e_{il})$,
\end{enumerate}
where $e_{ij}$ are the elementary matrices in $\uM_{m|n}(R)$. 

\item If $R$ is a local superalgebra, all of the automorphisms of
the superalgebra $\uM_{m|n}(R)$ are of the form:
$$
\begin{array}{ccc}
\rM_{m|n}(R) & \lra & \rM_{m|n}(R) \\ \\
(T, X) & \mapsto & TXT^{-1}
\end{array}
$$
for a suitable $T \in \rGL_{m|n}(R)$.

\item $\rAut(\uM_{m|n})$ is a closed subsuperscheme of 
$\rGL_{M|N}=\uspec k[x_{ij,kl}][d_1^{-1},d_2^{-1}]$,
 $M = m^2 + n^2$ and $N = 2mn$, defined by the equations:
\begin{equation} \label{closedsubsch}
\sum_k x_{ij,kk}=\de_{ij}, \qquad 
\sum_s x_{rs,ij}x_{st, kl}=\de_{jk}x_{rt,il},
\end{equation}
 where $\rGL_{M|N}(R)$ is identified with the parity preserving automorphisms of
the free $R$-module $\uM_{m|n}(R)$. 
\end{enumerate}

\end{lemma}

\begin{proof} (1).
If $\psi$ is an $R$-superalgebra endomorphism of 
$\uM_{m|n}(R)$ then the two relations are obviously satisfied and vice-versa.

\medskip\noindent
(2). Now assume $\psi$ is an automorphism of $\rM_{m|n}(R)$, $R$ local, 
which satisfies the relations $(a)$ and $(b)$.
We need to find $T \in \rGL_{m|n}(R)$ such that $\psi(e_{ij})=Te_{ij}T^{-1}$.
This is an application of super Morita theory (see \cite{kwok}), however
we shall recall the main idea to make this proof self-contained.
By $(a)$ and $(b)$ we have that 
$$
\sum \psi(e_{ii})=\id, \quad \psi(e_{ii})^2=\psi(e_{ii}), \quad
\psi(e_{ii})\psi(e_{jj})=0, \, i \neq j
$$
hence we can write
$$
R^{m|n}=\oplus \psi(e_{ii}) R^{m|n}
$$
Since by $(b)$ $\psi(e_{ji})\psi(e_{ii})=\psi(e_{ji})=\psi(e_{jj})\psi(e_{ji})$
we have that $\psi(e_{ji}): \psi(e_{ii}) R^{m|n} \lra \psi(e_{jj}) R^{m|n}$
(recall that $R$ is local so projective implies free). Hence there
exists a basis $\{t_i\}$ of the free module $R^{m|n}$ such that
$$
\psi(e_{ii}) R^{m|n} = {\mathrm{span}}_R \{t_i\}
$$
and $\psi(e_{ji}) t_i=t_j$. Let $T$ be the matrix whose columns are
the $t_i$'s, $T=\sum t_i \otimes e_i^*$, $T^{-1}=\sum e_i \otimes t_i^*$.
It is then immediate to verify $\psi(e_{ij})=T e_{ij} T^{-1}$.

\medskip\noindent
(3). This is immediate from (1). 
\end{proof} 

Let us view the multiplicative algebraic supergroup 
$\bG^{1|0}_m: \salg \lra \grps$ as the following subsupergroup of $\rGL_{m|n}$:
$$
\bG^{1|0}_m(R) = \{aI \, | \,  a \in R^*_0 \} \subset  \rGL_{m|n}(R).
$$
(Here $I$ denotes the identity matrix).

We shall not specify the definition on the arrows whenever
it is clear, as in this case.

\begin{definition}
We define 
the supergroup functor: $\widehat{\rPGL}_{m|n}:\salg \lra \grps$
$$
\widehat{\rPGL}_{m|n}(R)=\rGL_{m|n}(R)/\bG^{1|0}_m(R),
$$
\noindent and we call its sheafification (as customary) $\rGL_{m|n}/\bG^{1|0}$.
\end{definition}

We wish to show that $\rGL_{m|n}/\bG^{1|0}$ is representable and 
coincides with the projective linear supergroup, that is with the automorphism
supergroup of supermatrices.

\begin{definition}
We say that a functor $F: \salg \lra \grps$ 
is \textit{stalky} if
for any superalgebra $R$, the natural map 
\begin{align*}
\varinjlim_{f \notin \mathfrak{p}} F(R_f) \, \lra \, F(R_\mathfrak{p})
\end{align*}
is an isomorphism for any prime ideal $\mathfrak{p} \in R_0$.
\end{definition}

The next two lemmas  
are standard and their proof is the same
as in the ordinary case, see 
\cite{yi}.

\begin{lemma}\label{stalky}
$\rGL_{m|n}/\bG^{1|0}$ and $\rAut(\uM_{m|n})$ are stalky.
\end{lemma}

\begin{lemma}\label{localiso}
Let $\mathcal{F}, \mathcal{G}$ be stalky Zariski sheaves 
$\salg \to \grps$, 
$\alpha: \mathcal{F} \to \mathcal{G}$ a morphism. 
If $\alpha_{R}: \mathcal{F}(R) \to \mathcal{G}(R)$ 
is an isomorphism for all local superrings $R$, then $\alpha$ is an 
isomorphism of sheaves.
\end{lemma}

\begin{proposition}
The supergroup functor  $\rGL_{m|n}/\bG^{1|0}$ is representable and
is realized as the closed subsupergroup $\rAut(\uM_{m|n})$ of $\rGL_{M|N}$ for 
$M=m^2+n^2$ and $N=2mn$.
\end{proposition}

\begin{proof}
We need to establish an isomorphism of sheaves between 
 $\rGL_{m|n}/\bG^{1|0}$ and a closed subsupergroup of $\rGL_{M|N}$. 
We will first give a morphism
of sheaves and then show it is an isomorphism on local superalgebras;
since $\rGL_{m|n}/\bG^{1|0}$ is a stalky sheaf, this will be
enough.
We start by giving a morphism of presheaves 
$\widehat{\rPGL}_{m|n}$ and  $\rGL_{M|N}$; since $\rGL_{M|N}$ is a sheaf
then such a morphism will factor through the sheafification of
$\widehat{\rPGL}_{m|n}$ thus giving us a sheaf morphism.

\medskip
Consider the action of $\rGL_{M|N}$ on supermatrices
$\uM_{m|n}$, where $M=m^2+n^2$, $N=2mn$:
$$
\begin{array}{cccc}
\phi: & \rGL_{m|n}(R) \times \uM_{m|n}(R) & \lra & \uM_{m|n}(R) \\ \\
& (T, X) & \mapsto & TXT^{-1}
\end{array}
$$
This clearly factors through $\bG^{1|0}_m(R)$ hence gives a well defined action
$\rho$ of  $\widehat{\rPGL}_{m|n}$ and then in turn of $\rGL_{m|n}/\bG^{1|0}$ (see
comments at the beginning of the proof). 
Since $X \mapsto TXT^{-1}$, $T \in (\rGL_{m|n}/\bG^{1|0})(R)$ 
is a parity preserving $R$-superalgebra morphism, it
is immediate to verify we have a morphism of sheaves 
$$
\rGL_{m|n}/\bG^{1|0} \to \rAut(\uM_{m|n}).
$$
By the first part of Lemma \ref{matlemma}, $\rAut(\uM_{m|n})$ 
is represented by the closed subsuperscheme $H$ of 
$\rGL_{M|N}=\uspec k[x_{ij,kl}][d_1^{-1},d_2^{-1}]$ defined by the equations:
\begin{equation} \label{closedsubsch}
\sum_k x_{ij,kk}=\de_{ij}, \qquad 
\sum_s x_{rs,ij}x_{st, kl}=\de_{jk}x_{rt,il}
\end{equation}
\noindent  (Here $d_i$ denotes as usual the determinants of the 
diagonal blocks of 
indeterminates).
We want to show that the group homomorphism 
$(\rGL_{m|n}/\bG^{1|0})(R) \to [\rAut(\uM_{m|n})](R)$ is an isomorphism for $R$ local. 
$\psi \in \rGL_{M|N}(R)$ 
belongs to $H(R)$ if and only
its entries $\psi(e_{ij})_{kl}$ satisfy the above relations (\ref{closedsubsch})
(where in our convention $x_{ij,kl}$ corresponds to  $\psi(e_{ij})_{kl}$).
Hence by Lemma \ref{matlemma} we have the result for $R$ local. 
By Lemmas \ref{stalky} and \ref{localiso}, it is true for any superalgebra 
$R$ and this concludes the proof.
\end{proof}

\begin{remark}
The projective linear supergroup may also be obtained through the
Chevalley supergroup recipe as detailed in \cite{fg1, fg2, fg3}). It
corresponds to the choice of the adjoint action of the Lie superalgebra
${\mathfrak {sl}}_{m|n}$. In fact one may readily check that the Lie superalgebra
of $\rPGL_{m|n}$ is  indeed $\fsl_{m|n}$ and 
$(\rPGL_{m|n})_0=\rPGL_{m}\times \rPGL_n \times k^\times$. 
\end{remark}

\section{The automorphisms of the projective superspace}
\label{aut-sec}

We want to define the automorphism supergroup of the superscheme
$\bP^{m|n}$. 

\begin{definition}
We define the supergroup functor of \textit{automorphisms of the
projective superspace}:
$$
\Aut(\bP^{m|n})(A) := 
\Aut_A(\bP^{m|n} \times \uspec A)=
\Aut_A \bP^{m|n}_A, \quad A \in \salg.
$$
$\Aut(\bP^{m|n})$ is defined in an obvious way on the morphisms.
\end{definition}

The equality in the definition is straightforward noticing that
we can identify the $T$-points of $\bP^{m|n} \times \uspec A$ and 
of $\bP^{m|n}_A$. In fact
a $T$-point of $\bP^{m|n} \times \uspec A$ is a morphism $\phi: A \lra T$
and a morphism of $A$-modules via $\phi$, $L \lra T^{m|n}$. This is
exactly an element of  $\bP^{m|n}_A(T)$ and vice-versa.

\medskip
An automorphism $\psi \in \Aut_A \bP^{m|n}_A$ is
a family of automorphisms $\psi_T$ for all $T \in \salg_A$, which is functorial in $T$. 
$\psi_T: \bP^{m|n}_A(T) \lra \bP^{m|n}_A(T)$ must assign to a $T$-point
of $\bP^{m|n}_A(T)$, that is a morphism
$\al:L \lra T^{m|n}$, another morphism
$\al':L' \lra T^{m|n}$, where $L$, $L'$ are projective
rank $1|0$ $T$-modules, 
where the morphisms are interpreted as $A$-module morphisms. 
Similarly for the other characterizations of $T$-points as in Prop.
\ref{fopts-proj}.


We are now ready to relate the supergroup scheme $\rPGL_{m|n}$
with the automorphisms of $\bP^{m-1|n}$.

\begin{proposition}\label{pgl-embed} 
There is an embedding of supergroup functors  
$\rPGL_{m|n} \hookrightarrow \Aut(\bP^{m-1|n})$.
\end{proposition}

\begin{proof}
We first establish a morphism $\phi': \rGL_{m|n} \lra \Aut(\bP^{m-1|n})$.
If $X \in \rGL_{m|n}(A)$ and $\al \in \bP^{m-1|n}_A(T)=\{T^{m|n} \lra L\}\,/\,
\sim$,
$\psi: A \lra T$ 
we define
$$
\phi'(X)=\al \circ \rGL_{m|n}(\psi)(X)
$$
Clearly $\phi'$ factors through $\bG_m(A)$.
Since $\Aut(\bP^{m-1|1})$ is a sheaf, we have defined a morphism
$$
\phi:  \rPGL_{m|n} \lra \Aut(\bP^{m-1|n})
$$
The injectivity is clear. 
\end{proof}

\begin{remark} In general we
cannot expect to get an isomorphism between $\rPGL_{m|n}$ and
$\Aut(\bP^{m-1|n})$ for $n>1$ and this
is because of the peculiarity of the odd elements. Let us see
this in a simple example: $\bP^{1|2}$. Consider the morphism 
$\phi \in \bP_A^{1|2}$ given on the affine
pieces $U_0=\uspec A[u,\mu_1,\mu_2]$ and $U_1= A[v,\nu_1,\nu_2]$ by  
$$
\phi|_{U_0}(u,\mu_1,\mu_2)=(u+\mu_1\mu_2, \mu_1,\mu_2), \qquad
\phi|_{U_1}(v,\nu_1,\nu_2)=(v-\nu_1\nu_2, \nu_1,\nu_2)
$$
 
As $\phi$ is invertible, $\phi \in \Aut(\bP^{m|n})(A)$,
but is not obtained through an element of
$\PGL_{2|2}(A)$. In fact the coefficient in $\phi|_{U_0}$ of 
$\mu_1\mu_2 $ 
in an automorphism induced by a $\PGL_{2|2}(A)$  
transformation must be a nilpotent. Hence $\phi \not\in \PGL_{2|2}(A)$.
\end{remark}

We now want to show that we have an isomorphism between the projective linear
supergroup and the automorphism of the super projective when
$n=1$. The argument we give follows along the lines 
of the calculation of $\rAut(\bP^n)$ given in  \cite{ha} Ch. 2, Sec. 7. 

\begin{proposition}\label{pgl-iso} 
We have an isomorphism of supergroup functors:
$$
\PGL_{m+1|1} \cong \Aut(\bP^{m|1} )
$$
In particular, $\Aut(\bP^{m|1} )$ is a supergroup scheme.
\end{proposition}

\begin{proof}
Proposition \ref{pgl-embed} gives us an embedding of supergroup
functors $\PGL_{m+1|1}$ $\hookrightarrow$ $\Aut(\bP^{m|1})$. Now
let $f \in \Aut(\bP^{m|1}_A)$ and let $g$ be its inverse.
We want to show $f \in \PGL_{m+1|1}(A)$.
The automorphism $f$ induces the two line bundle morphisms 
$f^*\cO_A(1) \lra \cO_A(1)$ and  $g^*\cO_A(1) \lra \cO_A(1)$,
where $\cO_A(1):= p_1^*(\cO(1))$,
$p_1:\bP_A^{m|1} \lra \bP^{m|1}$ being the natural projection.
By Prop. \ref{linebundles-obs}, we know that $f^*\cO_A(1)=\cO(k) \otimes \cL_f$
and $g^*\cO_A(1)=\cO(l) \otimes \cL_g$. Let us  choose a suitable open
cover of $A$ in which both $\cL_f$ and $\cL_g$ are trivial. By a common 
abuse of notation we shall still write $A$ to denote 
the ring of global sections of an element of the open
cover, so we in fact are replacing $A$ with its localization. 
With such a choice we have
$f^*\cO_A(1)\cong\cO_A(k)$, $g^*\cO_A(1)\cong\cO_A(l)$. Since $f$ and $g$
are mutually inverse, we have:
$$
\cO_A(1)=(f^*\circ g^*)(\cO_A(1))=f^*( g^*(\cO_A(1)))=
f^*(\cO_A(l))=\cO_A(kl).
$$
Hence $kl=1$, whence $k = l = 1$, because for $k=l=-1$ 
we do not have global sections.

\medskip
So $f^*(\cO(1))\cong \cO(1)$, and choosing an isomorphism $F: f^*(\mathcal{O}(1)) \to \mathcal{O}(1)$ yields an isomorphism of
the global sections $\Gamma(\bP^m, f^*\cO_A(1))\cong \Gamma(\bP^m, \cO_A(1))$
By composing such an isomorphism with the natural isomorphism
$$
f^*: \Gamma(\bP^m, \cO_A(1)) \lra \Gamma(\bP^m, f^*\cO_A(1))
$$ 
we obtain
an $A$-linear automorphism
$$
T_F: \Gamma(\bP^m, \cO_A(1)) \lra \Gamma(\bP^m, \cO_A(1))
$$

\noindent and identifying $\Gamma(\bP^m, \cO_A(1))$ with $A^{m+1|1}$ we see that $T_F \in \rGL_{m+1|1}(A)$. However, $T_F$ depends on $F$. Suppose $G: f^*(\mathcal{O}(1)) \to \mathcal{O}(1)$ is another isomorphism, then $F^{-1} \circ G$ is an automorphism of $\mathcal{O}(1)$. Since $\underline{\Hom}(L, L) = L^* \otimes L = \mathcal{O}$ for any line bundle $L$, we see that an automorphism of $\mathcal{O}(1)$ is the same thing as an invertible even function on $\bP^{m|1}_A$,  and $F$ and $G$ differ by composing with multiplication by such a function. 

Therefore $f$ determines $T_F$ only up to multiplication by an 
invertible even function, i.e. $f$ uniquely determines an element 
$T := [T_F]$ of $\rPGL_{m+1|1}(A)$. 

Now 
in suitable coordinates we have that $T$ induces 
(up to scalar multiplication) an automorphism of the
$\Z$-graded superalgebra $A[z_0, \dots , z_m, \zeta]$. 
We leave to the reader the check that $\phi(T)$ is indeed $f$.
\end{proof}

\section{The SUSY-preserving automorphisms of $\bP^{1|1}_k$}  
\label{susy-sec}

In this section we want to consider those automorphisms of $\bP^{1|1}_k$ 
which preserve its unique (up to isomorphism) SUSY structure. 
For all of the standard notation of supergeometry refer to
\cite{ccf}.

\medskip
Let $k$ be our ground field, $\mathrm{char}(k)\neq 2$, $k$ algebraically
closed. All algebraic supergroups discussed below will be algebraic 
supergroups over $k$. 

\medskip
We recall that if, $X$ is a smooth algebraic supervariety over $k$ 
of dimension $1|1$, we define
a {\it SUSY structure} on $X$ as a $0|1$ distribution $\cD$ on 
$X$ such that the Frobenius map
\begin{align*}
\cD \otimes \cD &\lra TX/\cD\\
Y \otimes Z &\mapsto [Y,Z] \text{ mod } \mathcal{D}
\end{align*}

\noindent is an isomorphism (see, for example, \cite{ma1} for the definition of SUSY-structure in the complex analytic case). 
If $X \to S$ is a smooth family of algebraic 
supervarieties of relative dimension $1|1$ over an algebraic $k$-supervariety 
$S$, then the notion of relative SUSY structure may be defined in the 
analogous way, as a relative distribution in the relative tangent sheaf $TX/S$. 
In this case we say that $X \to S$ is a \textit{relative SUSY family}.

Our discussion is based on \cite{witten}.

\medskip
Let us start by interpreting $\bP^{1|1}_k$  
as a homogeneous superspace. Let $\underline{k}^{2|1}=(k^2, 
\cO_{\underline{k}^{2|1}})$  denote the 
affine superspace canonically associated to the $k$-super vector space 
$k^{2|1}$. Let us consider the action of the algebraic group 
$\underline{k}^\times$ on $\underline{k}^{2|1} \setminus \{0\}$, 
given in the functor of points notation by: 

$$
t \cdot (z_0,z_1, \zeta)=(tz_0,tz_1, t\zeta)
$$
Consider the projection (as topological map): 
$$
\pi:k^2\setminus \{0\}\lra
k^2 \setminus \{0\}/\, k^\times \cong \bP^1
$$
Define the sheaf on the topological space $\bP^1_k$ consisting of the 
$\underline{k}^\times$-invariant sections:
$$
\cF(U):=\cO_{\underline{k}^{2|1}}(\pi^{-1}(U)))^{\underline{k}^\times}
$$
One can readily check that  $(\bP^1_k, \cF)$ is the superscheme
$\bP^{1|1}_k$ as defined in Sec. \ref{pmn-sec}. 

\medskip
Let $z_0$, $z_1$, $\zeta$ be global coordinates on  
$\underline{k}^{2|1}$.
We now consider the Euler vector field 
$E= z_0 \partial_{z_0}+z_1 \partial_{z_1}+ \zeta \partial_{\zeta}$,
which represents (in the chosen coordinates) the infinitesimal
generator for the $\underline{k}^\times$ action on 
$\underline{k}^{2|1} \setminus \{0\}$.
Since $E$ is everywhere nonsingular, it generates a trivial $1|0$
line bundle.
As in the classical case, we have the Euler exact sequence of vector 
bundles on $\bP^{1|1}_k$: 
\begin{equation} \label{eulerseq}
0 \to \mathcal{O}^{1|0} \xrightarrow{i} 
\mathcal{O}(1) \otimes \Der(S) \xrightarrow{j} T\mathbb{P}^{1|1}_k \to 0
\end{equation}
where $i$ is the inclusion of the trivial $1|0$ line bundle 
$\langle E \rangle$ with global basis the Euler vector field. Here $\Der(S)$ is the $k$-super vector space of $k$-linear derivations on $S := 
\underline{\mathrm{Sym}}((k^{2|1})^*)$; it has as basis the derivations 
$\partial_{z_i}, \partial_\zeta$. Thus $\mathcal{O}(1) \otimes \Der(S)$ is the sheaf 
whose sections on $U$ are the linear vector fields on $\pi^{-1}(U)$. Any local section of $\mathcal{O}(1) \otimes \Der(S)$ induces a corresponding local $k$-linear derivation on $\mathcal{O}_{\mathbb{P}^{1|1}_k}$ by restricting it to act on $\underline{k}^\times$-invariant functions; this defines $j$. Injectivity of $i$ and the inclusion $\im(i) \subseteq \Ker(j)$ follow from the fact that $E$ is nonsingular and the infinitesimal generator for the $\underline{k}^\times$-action; a standard calculation in the usual affine cells shows that $\Ker(j) \subseteq \im(i)$ and that $j$ is surjective. Note that the sequence continues to remain exact on $\bP^{1|1}_A$ after base change to any affine $k$-supervariety $\underline{\spec}(A)$, with $T\bP^{1|1}_k$ replaced by the relative tangent bundle $T\bP^{1|1}_A/\spec(A)$. We will denote the $A$-superalgebra $S \otimes_k A$ by $S_A$. 

\medskip
We now come to the SUSY structure. 

\begin{definition}
Let $(X \to S, \mathcal{D})$ be a relative SUSY family. An $S$-auto\-morphism $f: X \to X$ is {\it SUSY structure-preserving} (or simply {\it SUSY-pre\-serving}) if and only if $(df_p)(\mathcal{D}_p) = \mathcal{D}_{f(p)}$ for any $p \in X$.
\end{definition}

We will consider SUSY structures given by sections of $\mathcal{O}_A(1) \otimes \Omega_{S/A}$. Here $\Omega_{S/A}$ denotes the $A$-module of K{\"a}hler differentials on $S_A$, i.e. the $A$-dual to $\Der(S_A)$; it has as basis the differentials $dz_i, d\zeta$. When we speak of the kernel of a section $\omega$ of $\mathcal{O}_A(1) \otimes \Omega_{S/A}$, we mean the kernel of $\omega$ when $\omega$ is interpreted as a morphism of sheaves of $\mathcal{O}_{\bP^{1|1}_A}$-modules 
from $\cO_A(1) \otimes \Der(S_A) \to \cO_A(2)$.

\begin{proposition}
Let $s:=z_1 \, dz_0 - z_0 \, dz_1 - \zeta \, d \zeta$. Then the image of $\Ker(s)$ under $j$ is a SUSY structure on $\bP^{1|1}_k$.
\end{proposition}

\begin{proof}
In the affine open subsupervariety 
$U_1 := \{z_1 \neq 0\} 
\subset \bP^{1|1}_k$, one calculates that the Euler vector field $E$ and 
the linear vector field $\widehat{Z}_1 = \zeta \partial_{z_0} + z_1 
\partial_\zeta$ lie in $\Ker(s)$ and are linearly independent. 
At any point $p \in \bP^{1|1}_k$, $s$ induces a linear map of 
super vector spaces $s_p: [\mathcal{O}(1) \otimes \Der(S)]_p \to 
[\mathcal{O}(2)]_p$ on the fibers. It is clear that $s$ is a basepoint-free 
section, hence $s_p$ is always surjective. By linear algebra, 
$\Ker(s_p)$ is $1|1$ dimensional and hence $E_p$ and $\widehat{Z_1,}_p$ 
span $\Ker(s_p)$. By the super Nakayama's lemma, $E$ and 
$\widehat{Z_1}$ span $\Ker(s)$ near $p$. Since $p$ was arbitrary, $E$ and 
$\widehat{Z_{1}}$ form a basis for $\Ker(s)$ in $U_1$.

One sees
that $Z_{1} := j(\widehat{Z_{1}}) = 
\partial_\eta + \eta \partial_w$, where $w = z_0/z_1, \eta = \zeta/z_1$ are the usual affine coordinates in $U_{1}$. $Z_{1}^2 = \partial_w$ and so 
$Z_{1}$ defines a SUSY structure in $U_1$. A similar calculation with the linear vector field $\widehat{Z}_0 := -\zeta \partial_{z_1} + z_0 \partial_\zeta$ shows that $j(\Ker(s))$ defines a SUSY structure on $U_0 = \{z_0 \neq 0\}$, hence the image of $\Ker(s)$ under $j$ defines a SUSY structure on $\bP^{1|1}_k$.
\end{proof}

We note that by the considerations of \cite{fk}, this is the unique SUSY structure on $\bP^{1|1}_k$, up to SUSY-isomorphism.

We now need the following proposition. The
proof is completely similar to the one in \cite{fk}  Prop. 5.2, however
since the context here is more general, we include it for completeness.

\begin{lemma}\label{SUSYform}
Let $A$ be an affine $k$-superalgebra. Let $\omega, \omega'$ be two global sections of $\cO_A(1) \otimes \Omega_{S/A}$ such that $\mathcal{D} := j(\Ker(\omega)), \mathcal{D}' := j(\Ker(\omega'))$ are $0|1$ distributions on $\bP^{1|1}_A$. Suppose $\mathcal{D} = \mathcal{D}'$. Then $\omega' = h \omega$ for some even invertible function $h$ on $\bP^{1|1}_A$.
\end{lemma}

\begin{proof}
Let $p \in \bP^{1|1}_A$ be a point. 
Since $\mathcal{D}$ is \textit{locally a direct summand of  $T \bP^{1|1}_A/\underline{\spec}(A)$ }, 
we have a local splitting 
$\mathcal{D}|_U \oplus \mathcal{E} = (T\bP^{1|1}_A/\underline{\spec}(A))|_U$ 
in some neighborhood $U \ni p$.
Via the Euler exact sequence (base changed to $\spec(A)$), we may lift $\mathcal{D}|_U$ (resp. $\mathcal{E}$) uniquely to a rank $1|1$ (resp. $2|0$) submodule $\widehat{\mathcal{D}}$ (resp. $\widehat{\mathcal{E}}$) of $[\cO_A(1) \otimes \Der(S_A)]|_U$ containing the $1|0$ line bundle $\langle E \rangle$ spanned by the Euler vector field, such that $\widehat{\mathcal{D}} \cap \widehat{\mathcal{E}} = \langle E \rangle$. We may therefore find local sections $\widehat{Z}$ (resp. $\widehat{X}$) of $\widehat{\mathcal{D}}$ (resp $\widehat{\mathcal{E}}$) such that $\widehat{Z}, E$ (resp. $\widehat{X}, E$) form a basis for $\widehat{\mathcal{D}}$ (resp. $\widehat{\mathcal{E}}$). Note that the condition $\widehat{\mathcal{D}} \cap \widehat{\mathcal{E}} = \langle E \rangle$ implies $\widehat{X}, \widehat{Z}, E$ form a basis of $[\cO_A(1) \otimes \Der(S_A)]|_U$.

Viewing $\omega|_U$ as an $\cO_{\bP^{1|1}_A}$-linear map from $[\cO_A(1) \otimes \Der(S_A)]|_U$ to $\cO_A(2)|_U$, we have an induced linear map of super vector spaces 

\begin{equation*}
\omega_p: (\cO_A(1) \otimes \Der(S_A))_p \to  (\cO_A(2))_p.  
\end{equation*}

\smallskip

As $\Ker(\omega_p) =  span \{\widehat{Z}_p, E_p\}$, we see by linear algebra that $\omega_p$ is a surjection, and that $\omega_p(\widehat{X}_p)$ is a basis for $(\cO_A(2))_p$; the analogous conclusion holds for $\omega'_p$ and $\omega'_p(\widehat{X}_p)$. Hence by the super Nakayama's lemma, $\omega(\widehat{X})$ is a basis for $\cO_A(2)|_U$, and the same is true of $\omega'(\widehat{X})$ (shrinking $U$ if necessary). Hence $\omega'(\widehat{X})/\omega(\widehat{X})$ is an invertible even function on $U$; let us denote it by $h$.

To show that $h$ is independent of the local complement $\mathcal{E}$ and the choice of basis element $\widehat{X}$, suppose $\mathcal{E}'$ is another local complement to $\mathcal{D}$ on $U$, and let $\widehat{X}', E$ be a basis of the lift $\widehat{\mathcal{E}}'$ of $\mathcal{E}'$. Then we have $\widehat{X}' = a\widehat{X} + bE + \alpha \widehat{Z}$ for some $a, b, \alpha \in \cO_{\bP^{1|1}_A}(U)$, $a, b$ even and $\alpha$ odd. As $\widehat{X}, E, \widehat{Z}$ and $\widehat{X}', E, \widehat{Z}'$ are both local bases for $\cO_A(1) \otimes \Der(S_A)$, $a$ must be a unit.

Then we have 
\begin{align*}
\omega'(\widehat{X}')/\omega(\widehat{X}') &= \omega'(a\widehat{X} + bE + \alpha \widehat{Z})/\omega(a\widehat{X} + bE + \alpha \widehat{Z})\\
&= \omega'(\widehat{X})/\omega(\widehat{X})
\end{align*}

\noindent since $\omega, \omega'$ both annihilate $E$ and $\widehat{Z}$. This proves that the expression \break $\omega'(\widehat{X})/\omega(\widehat{X})$ is independent of all choices and hence $h$ is a well-defined function on all of $\bP^{1|1}_A$. The equality $\omega' = h \omega$ clearly holds locally, and since $h$ is now known to be globally defined, it holds globally.

\end{proof}

\begin{proposition}
Let $f$ be an automorphism of $\bP^{1|1}_A$. Then $f$ preserves the SUSY structure defined by $s$ if and only if for some (hence every) lift $\widetilde{f}$ of $f$ to $\rGL_{2|1}(A)$, 
$\widetilde{f}^*(s) = t s$ for some invertible function $t$. 
\end{proposition}

\begin{proof}
We begin by noting that $\rGL_{2|1}(A)$ preserves $A^*_0$-invariant open subsets of $\mathbb{A}^{2|1}_A \backslash \{0\}$, hence it acts naturally by pullback of functions on $\cO_A(1) \otimes \Der(S_A)$, where we interpret the latter as the sheaf assigning to any open subset $U \subseteq \bP^{1|1}_A$ the linear vector fields on $\pi^{-1}(U) \subseteq \mathbb{A}^{2|1}_A \backslash \{0\}$. 

The subsupergroup of invertible scalar matrices $\{cI : c \in A^*_0\}$ is central in $\rGL_{2|1}(A)$, hence this $\rGL_{2|1}(A)$-action preserves the subalgebra of $A^*_0$-invariant functions on any $A^*_0$-invariant open subset of $\mathbb{A}^{2|1}_A \backslash \{0\}$. Hence we have an induced $\rGL_{2|1}(A)$-action on the sheaf $\cO_{\bP^{1|1}_A}$. Clearly, invertible scalar matrices act trivially on $\cO_{\bP^{1|1}_A}$, hence the $\rGL_{2|1}(A)$-action on $\cO_{\bP^{1|1}_A}$ factors through $\rPGL_{2|1}(A)$.

We see from the above that the action of $\rGL_{2|1}(A)$ on $\cO_A(1) \otimes \Der(S_A)$ by pullback of functions induces naturally a $\rPGL_{2|1}(A)$-action on $\mathcal{O}_{\bP^{1|1}_A}$, hence on $T\bP^{1|1}_A/\underline{\spec}(A)$, also given by pullback of functions. But this is precisely the $\rPGL_{2|1}(A)$-action on $T \bP^{1|1}_A/\underline{\spec}(A)$  induced by the action of $\rPGL_{2|1}(A)$ on $\bP^{1|1}_A$ by automorphisms. 

Recalling that the sheaf morphism $j: \cO_A(1) \otimes \Der(S_A) \to T\bP^{1|1}_A/\underline{\spec}(A)$ is just given by restricting a linear vector field to act on $A^*_0$-invariant functions, we see $j$ is equivariant with respect to the $\rGL_{2|1}(A)$- and $\rPGL_{2|1}(A)$-actions previously defined.

 We also have a $\rGL_{2|1}(A)$-action on $\cO_A(1) \otimes \Omega_{S/A}$ by the natural action on both factors, and for $\omega \in \Gamma(\cO_A(1) \otimes \Omega_{S/A}) = \Gamma(\cO_A(1)) \otimes \Omega_{S/A}$, we write $g^*(\omega)$ for $g \cdot \omega$.


Since the action of $\rGL_{2|1}(A)$ on $\cO_A(1) \otimes \Der(S_A)$ is the same as the natural action on the individual factors, and the $\rGL_{2|1}(A)$-action on $\Omega_{S/A}$ is dual to that on $\Der(S_A)$, it follows that the evaluation pairing $[\cO_A(1) \otimes \Der(S_A)] \otimes [\cO_A(1) \otimes \Omega_{S/A}] \to \cO_A(2)$ is $\rGL_{2|1}(A)$-equivariant, where $\cO_A(2)$ is endowed with the natural $\rGL_{2|1}(A)$-action.

From the preceding discussion, we see that $f$ is SUSY-preserving if and only if $j[\Ker(\omega)]_p = j[\Ker(\widetilde{f}^*(\omega)]_p$ for any point $p$. 

We claim this is true if and only if $j[\Ker(\omega)] = j[\Ker(\widetilde{f}^*(\omega))]$.  One direction is clear. For the other, suppose $j[\Ker(\omega)]_p = j[\Ker(\widetilde{f}^*(\omega))]_p$ for any point $p$. Then by the super Nakayama's lemma $j[\Ker(\omega)] = j[\Ker(\widetilde{f}^*(\omega))]$ in a neighborhood of $p$, hence globally. The claim then follows from Lemma \ref{SUSYform}.
\end{proof}

In order to determine the supergroup of SUSY-preserving automorphisms of $\mathbb{P}^{1|1}_k$ we must discuss various other supergroups. We follow closely the discussion in \cite{ma1}. 

\begin{definition}
The $2|1$-dimensional {\it conformal symplecticorthogonal supergroup} 
$\rC_{2|1}$ is the subfunctor of $\rGL_{2|1}$ that preserves, up to multiplication by an even invertible constant, the split nondegenerate supersymplectic form on $k^{2|1}$ given by $(v, w) = v^tHw$, where

\begin{equation}
H := \begin{pmatrix} 0 & 1 & 0 \\ -1 & 0 & 0 \\ 0 & 0 & -1 \end{pmatrix},
\end{equation}\ 

\noindent and ${}^t$ denotes the super transpose of a matrix. 
More precisely, for every 
$k$-superalgebra $A$, $\rC_{2|1}$ 
is the functor 
$\salgk \to \grps$ given by

\begin{equation}\label{defnC21}
\rC_{2|1}(A) := \{B \in \rGL_{2|1}(A) \, : \, B^tHB = Z(B)H\},
\end{equation}

\noindent where $Z: \rGL_{2|1} \to \mathbb{G}^{1|0}_m$ is a fixed homomorphism.

The $2|1$-dimensional 
{\it projective conformal symplecticorthogonal supergroup}
$\rPC_{2|1}$ is the image of $\rC_{2|1}$ in $\rPGL_{2|1}$, i.e, it is the sheafification of the group-valued functor $A \to \rC_{2|1}(A)/\{aI : a \in A^*_0\}$.
\end{definition}

\begin{proposition}
$\rC_{2|1}$ and $\rPC_{2|1}$ are representable.
\end{proposition}

\begin{proof}
Taking the Berezinian of both sides of (\ref{defnC21}), one sees that $Z(B) = \ber(B)^2$. Thus, given 

\begin{equation*}
B = \begin{pmatrix} a & b & \alpha  \\ c & d & \beta \\ \gamma& \delta & e \end{pmatrix} \in \rGL_{2|1}(A),
\end{equation*}\

\noindent a direct calculation shows that $B$ satisfies (\ref{defnC21}) 
if and only if the following equations hold: 

\begin{align*}
&e^2 + 2 \alpha \beta = \ber(B)^2\\
&ad - bc - \gamma \delta = \ber(B)^2\\
&a \beta -c \alpha - e \gamma = 0\\
&b \beta - d \alpha - e \delta = 0.
\end{align*}

Thus these equations define $\rC_{2|1}$ as a closed affine algebraic 
subsupergroup of $\rGL_{2|1}$. 

To prove that $\rPC_{2|1}$ is representable, we use the trick of \cite{ma1}. Let $\rSC_{2|1}$ denote the functor $\salgk \to \grps$ given by

\begin{equation*}
\rSC_{2|1}(A) := \{B \in \rC_{2|1}(A) :  \ber(B) = 1\}.
\end{equation*}\

Since its defining equations are those of $\rC_{2|1}$ together with the equation $\ber(B) = 1$, $\rSC_{2|1}$ is a closed affine algebraic subsupergroup 
of $\rGL_{2|1}$. There is a short exact sequence of supergroups

\begin{align}\label{splittingconformal}
0 \to \rSC_{2|1} \to \rC_{2|1} \xrightarrow{\ber} \mathbb{G}^{1|0}_m \to 0.
\end{align}\

\noindent There is a splitting of this sequence, given on $A$-points by sending $a \in A^*_0$ to $aI$, and the image of $\bG^{1|0}_m$ under the splitting 
is clearly normal in $\rC_{2|1}$, hence $\rC_{2|1}$ is the internal direct 
product of $\rSC_{2|1}$ and the subsupergroup $\{aI : a \in A^*_0\}$. This direct product decomposition allows us to naturally identify the functor $\rPC_{2|1}$ with the functor of points of $\rSC_{2|1}$; in particular, we see $\rPC_{2|1}$ is an affine algebraic supergroup, isomorphic to $\rSC_{2|1}$.
\end{proof}

\begin{definition}
The $2|1$-dimensional {\it symplecticorthogonal supergroup} $\rSpO_{2|1}$  
is the functor $\salgk \to \grps$ 

\begin{equation}\label{defnOSp}
\rSpO_{2|1}(A) := \{B \in \rGL_{2|1}(A) \, : \, B^tHB = H\}.
\end{equation}\
\end{definition}

\begin{remark}
$\rSpO_{2|1}$  is well-known to be representable; the reader may readily write down defining equations for $\rSpO_{2|1}$, completely analogous to those for $\rC_{2|1}$, which show that $\rSpO_{2|1}$ is a closed affine algebraic subsupergroup of $\rGL_{2|1}$.
\end{remark}

\begin{proposition}\label{irredcomp}
$\rPC_{2|1}$ is isomorphic to the irreducible component $(\rSpO_{2|1})^0$ of $\rSpO_{2|1}$ containing the identity. 
\end{proposition}

\begin{proof}
Taking the Berezinian of both sides of (\ref{defnOSp}) shows that $\ber(B) = \pm 1$ for any $B \in \rSpO_{2|1}(A)$. This yields a short exact sequence of supergroups

\begin{align}
0 \to \rSC_{2|1} \to \rSpO_{2|1} \xrightarrow{\ber} \{\pm 1\} \to 0.
\end{align}\

\noindent which is split by the morphism $\pm 1 \mapsto \pm I$ and $\{\pm I\}$ is obviously normal in $\rSpO_{2|1}$. Thus $\rSpO_{2|1}$ is the internal direct product of $\{\pm I\}$ and $\rSC_{2|1}$. Note that $\rSC_{2|1}$ is irreducible (one sees from its defining equations that its reduced algebraic group is $\rSL_2$, which is known to be irreducible). Let $(\rSpO_{2|1})^0$ denote the irreducible component of $\rSpO_{2|1}$ that contains the identity. We claim $\rSC_{2|1} = (\rSpO_{2|1})^0$. Since $I \in \rSC_{2|1} \cap (\rSpO_{2|1})^0$, it is clear $\rSC_{2|1} \subseteq (\rSpO_{2|1})^0$. Conversely, we see that $(\rSpO_{2|1})^0 \subseteq \rSC_{2|1}$: the restriction of the morphism $\ber$ to the irreducible supervariety $(\rSpO_{2|1})^0$ must be constant, hence equal to $1$. Since we previously showed $\rPC_{2|1}$ is isomorphic to $\rSC_{2|1}$, the proposition is proven. 
\end{proof}

\begin{theorem}
The algebraic supergroup $\rAut_{\SUSY}(\bP^{1|1}_k)$ of SUSY preserving automorphisms of $\bP^{1|1}_k$ is isomorphic to $(\rSpO_{2|1})^0$.
\end{theorem}

\begin{proof}
As $\rAut_{\SUSY}(\bP^{1|1}_k)$ is a sheaf, the theorem reduces to the case of calculating $\rAut_{\SUSY}(\bP^{1|1}_k)(A)$ where $A$ is a 
$k$-superalgebra. For this, we note that $\mathbb{P}^{1|1}_A$ has the SUSY structure over $A$ induced by base change from $\mathbb{P}^{1|1}_k$, given by $s$.

Let $g \in \PGL_{2|1}(A)$ be an automorphism of $\mathbb{P}^{1|1}_A$, and 
$\tg$ a lift of $g$ to $\rGL_{2|1}(A)$. Recall that we have a natural action of the group of $A$-points of $\rGL_{2|1}(A)$ on $\Gamma(\cO_A(1) \otimes \Omega_{S/A})$. More concretely, in the given coordinates we have for any matrix $\tg \in \rGL_{2|1}(A)$:
$$
\tg \cdot \begin{pmatrix} z_0 \\ z_1 \\ \zeta \end{pmatrix} \,=\, 
 \tg\begin{pmatrix} z_0 \\ z_1 \\ \zeta \end{pmatrix}, \qquad
\tg \cdot \begin{pmatrix} dz_0 \\ dz_1 \\ d\zeta \end{pmatrix} \,=\,
\tg \begin{pmatrix} dz_0 \\ dz_1 \\ d\zeta \end{pmatrix} \qquad
$$
where we write $z_i$ for $z_i \otimes 1$ and so on. 

By Prop. \ref{SUSYform}, $g$ is SUSY-preserving if and only if $\tg$ sends
$$
s=z_1 dz_0 - z_0 dz_1 - \zeta d \zeta=
\begin{pmatrix} z_0 & z_1 & \zeta \end{pmatrix} 
H \begin{pmatrix} dz_0 \\ dz_1 \\ d\zeta \end{pmatrix}, \qquad
H= \begin{pmatrix} 0 & 1 & 0 \\ -1 & 0 & 0 \\ 0 & 0 & -1 \end{pmatrix},
$$
to a multiple of $s$ by an invertible even function. Hence
$$
\begin{pmatrix} z_0 & z_1 & \zeta \end{pmatrix} \tg\,
{}^tH\tg
\begin{pmatrix} dz_0 \\ dz_1 \\ d\zeta \end{pmatrix}=
\begin{pmatrix} z_0 & z_1 & \zeta \end{pmatrix} Z(\tg)H
\begin{pmatrix} dz_0 \\ dz_1 \\ d\zeta \end{pmatrix},
$$
i.e. $\tg \in \rC_{2|1}(A)$. It follows from equation (\ref{splittingconformal}) that $g$ lies in $\rPC_{2|1}(A)$, 
which is naturally identified with $(\rSpO_{2|1})^0(A)$ by Prop. \ref{irredcomp}.
\end{proof}

\end{document}